\documentclass[a4paper, 12pt]{article}
\usepackage{amssymb}
\usepackage{amsmath}
\usepackage{amsthm}
\usepackage[dvipdfm]{graphicx}
\usepackage{verbatim,enumerate}
 \newtheorem{theorem}{Theorem}[section]
 \newtheorem*{theorem*}{Theorem}
 \newtheorem*{lemma*}{Lemma}
 \newtheorem{proposition}[theorem]{Proposition}
 
 \newtheorem{fact*}{Fact}
 \newtheorem{lemma}[theorem]{Lemma}
 \newtheorem{corollary}[theorem]{Corollary}
\theoremstyle{definition}
 
 \newtheorem{remark}[theorem]{Remark}
 \newtheorem*{remark*}{Remark}

\numberwithin{equation}{section}

\renewcommand{\labelenumi}{\theenumi}
\usepackage[usenames]{color}

\newcommand{\vect}[1]{\boldsymbol{#1}}
\renewcommand{\vec}[1]{\boldsymbol{#1}}
\newcommand{\R}{\boldsymbol{R}}
\newcommand{\Z}{\boldsymbol{Z}}

\newcommand{\rank}{\operatorname{rank}}

\newcommand{\image}{\operatorname{Im}}
\renewcommand{\phi}{\varphi}
\newcommand{\sgn}{\operatorname{sgn}}
\newcommand{\inner}[2]{\left\langle{#1},{#2}\right\rangle}

\newcommand{\ep}{\varepsilon}
\newcommand{\zv}{\vect{0}}

\newcommand{\e}{\vect{e}}

\newcommand{\x}{\vect{x}}
\newcommand{\y}{\vect{y}}
\newcommand{\z}{\vect{z}}
\newcommand{\dd}{\vect{d}}
\newcommand{\pmt}[1]{{\begin{pmatrix} #1  \end{pmatrix}}}
\newcommand{\mycomment}[1]{}
\newcommand{\firff}{\operatorname{\mathit{I}}}
\newcommand{\secff}{\operatorname{\mathit{I\!I}}}

\begin{document}
\begin{center}
{\large {\bf Geometric invariants of cuspidal edges}}
\\[2mm]
{\today}
\\[2mm]
\renewcommand{\thefootnote}{\fnsymbol{footnote}}
Luciana F. Martins and
Kentaro Saji\\
\footnote[0]{
Dedicated to Professor Mar\'ia del Carmen Romero-Fuster
  on 
  the occasion of her sixtieth birthday}

\begin{quote}
{\small
We give a normal form of the cuspidal edge
which uses only diffeomorphisms on the source
and isometries on the target.
Using this normal form, we study differential
geometric invariants of
cuspidal edges which determine them up to order three.
We also
clarify relations between these invariants.}
\end{quote}
\end{center}
\section{Introduction}
A generic classification of singularities of wave fronts was given
by Arnol'd and Zakalyukin. They showed that the generic
singularities of wave fronts in $\R^3$ are cuspidal edges and
swallowtails (see \cite{AGV}, for example). Recently, there are
numerous studies of wave fronts from the viewpoint of differential
geometry, for example \cite{izudual,mu,nishimura,front,suyojm}.
Cuspidal edges are fundamental singularities of wave fronts in
$\R^3$. The singular curvature and the limiting normal curvature for
cuspidal edges are defined in  \cite{front} by a limit of geodesic
curvatures and a limit of normal curvatures, respectively. On the
other hand, the umbilic curvature is defined in \cite{mn} for
surfaces in Euclidean 3-space with corank 1 singularities, by using
the first and second fundamental forms.  So, the umbilic curvature
is defined for cuspidal edges. It is shown in \cite{mn} that if the
umbilic curvature $\kappa_u$ is non-zero at a singular point, then
there exists a unique sphere having contact not of type $A_n$ (for
example, $D_4$, $E_6$ etc) with the surface in that point: the
sphere with center in the normal plane of the surface at the point,
with radius equal to $1/\kappa_u$ and in a well defined direction of
the normal plane.

Therefore, the singular, the limiting normal
and the umbilic curvatures are
invariants  defined by using
fundamental tools of differential geometry of
surfaces and singularity theory, and they are
fundamental invariants of cuspidal
edges. Needless to say, the curvature and torsion
of a cuspidal edge
locus as a space curve in $\R^3$ are also
fundamental invariants.

In this paper we clarify the relations amongst of these invariants
and also make a list of invariants which determine cuspidal edges up
to order three. We show that, in the case of cuspidal edges, the
umbilic curvature $\kappa_u$ coincides with the absolute value of
limiting normal
curvature $\kappa_n$ (Theorem \ref{thm:ku}). In this sense, the
umbilic curvature is a generalization of the normal curvature for
surfaces with corank 1 singularities. It should be remarked that the
umbilic curvature does not require a well-defined unit normal
vector, and it is meaningful as a geometric invariant of surfaces
with corank 1 singularities in general. We show that the singular
curvature $\kappa_s$ and the limiting normal curvature $\kappa_n$ at
a singular point of a surface $M$ in $\R^3$ with singularities
consisting of cuspidal edges are equivalent to the principal
curvatures of a regular surface in the following sense. When $M$ is
a regular surface in $\R^3$ given in the Monge form, that is, by the
equation $z=f(x,y)$ for some smooth function $f$, and its
first derivatives with respect to $x$ and $y$ vanishing at $(0,0)$
(or, equivalently, its tangent plane at the origin is given by
$z=0$), then, taking the $x$ and $y$ axes to be in principal
directions at the origin, the surface $M$ assumes the local  form
$$f(u,v) = \frac{1}{2}(\alpha_1 u^2 + \alpha_2 v^2) + h.o.t.,$$
where $\alpha_1, \alpha_2$ are the principal curvatures at the
origin and $h.o.t.$ represents terms whose degrees are greater than
two.

Now, if $M$ is  a surface in $\R^3$ with singularities consisting of
cuspidal edges, we show that  $M$ can be parametrized by an equation
of Monge type (just using changes of coordinates in the source and
isometries in the target, which do not change the geometry of the
surface), called here of ``normal form'', given by
$$
f(u,v) = \frac{1}{2}(2 u,\, \kappa_s u^2 + v^2 +h.o.t., \, \kappa_n
u^2 + h.o.t. )
$$
(see Section \ref{sec:west:type} and Theorem \ref{thm:rel1}).
Therefore, $\kappa_s$ and $\kappa_n$ can be considered as the
principal curvatures of $M$ at singular points. But we can say even
more about these two invariants. While for a  regular curve in a
regular surface it holds that $\kappa^2 = \kappa_n^2 +
\kappa_g^2,$ where $\kappa_n$ and $\kappa_g$ are the normal and
geodesic curvatures of the curve, respectively, and $\kappa$ is the
curvature of the curve as a space curve,  for the singular curve
consists of cuspidal edges the relation $$\kappa^2 =
\kappa_n^2+ \kappa_s^2$$
holds (see Corollary \ref{cor:rel:ku:ks}).

Furthermore, using the normal form, we detect (Section
\ref{sec:invar}) invariants up to order three, and show (Section
\ref{sec:main:th}) that the torsion of the curve consisting of
cuspidal edges as a space curve, $\kappa_s$, $\kappa_n$ and these
three invariants determine the cuspidal edge up to order three (see
Theorem \ref{thm:main}). In a joint work \cite{msuy} of M. Umehara,
K. Yamada and the authors, we consider intrinsic properties
 of these invariants and
the relation between boundedness of Gaussian curvature near
cuspidal edges.

The normal form of the Whitney umbrella (or cross-cap) was given by
J. M. West in \cite{west} and  it was shown to be very useful for
considering the differential geometry of surfaces near the singular
point. See also \cite{bw}.
For instance,  using this normal form,
  the authors in \cite{hhnuy} showed
that there are three fundamental intrinsic invariants for
cross-caps, and also the existence of extrinsic invariants is shown.
Some other works where this normal form was very important are
\cite{dt,fh,
ggs,hhnuy,mn,joey,Oset-Tari,faridtari}, for example. The
normal form for cuspidal edges is fundamental in this paper
for
finding geometric invariants of cuspidal edges, and the authors believe
that this normal form can be used for other problems
similar to those considered in the references mentioned just above.
\section{Preliminaries}\label{sec:prelim}

The unit cotangent bundle $T^*_1\R^{3}$ of  $\R^{3}$ has the
canonical contact structure and can be identified with the unit
tangent bundle $T_1\R^{3}$. Let $\alpha$ denote the canonical
contact form on it. A map $i:M\to T_1\R^{3}$ is said to be {\em
isotropic\/} if $\dim M=2$ and the pull-back $i^*\alpha$ vanishes
identically. An isotropic immersion is called a {\em Legendrian
immersion\/}. We call
the image of $\pi\circ i$
the {\em wave front set\/}
 of $i$,
where $\pi:T_1\R^{3}\to\R^{3}$ is the canonical projection and we
denote  it by\/ $W(i)$. Moreover, $i$ is called the {\em Legendrian
lift\/} of $W(i)$. With this framework, we define the notion of
fronts as follows: A map-germ $f:(\R^2,\zv) \to (\R^{3},\zv)$ is
called a
 {\em wave front\/} or a {\em front\/}
 if there exists a unit vector field $\nu$ of $\R^{3}$ along $f$
 such that
  $L=(f,\nu):(\R^2,\zv)\to (T_1\R^{3},\zv)$ is
 a Legendrian immersion by an identification $T_1\R^3 = \R^3 \times S^2$, where $S^2$
 is the unit sphere in $\R^3$ (cf. \cite{AGV}, see also \cite{krsuy}).
A point $q\in (\R^2, \zv)$ is a singular point if $f$ is not an
immersion at $q$.

A singular point $p$ of a map $f$ is called a {\em cuspidal edge\/}
if the map-germ $f$ at $p$ is $\mathcal{A}$-equivalent to
$(u,v)\mapsto(u,v^2,v^3)$ at $\zv$. (Two map-germs
$f_1,f_2:(\R^n,\zv)\to(\R^m,\zv)$ are $\mathcal{A}$-{\em
equivalent}\/ if there exist diffeomorphisms
$S:(\R^n,\zv)\to(\R^n,\zv)$ and $T:(\R^m,\zv)\to(\R^m,\zv)$ such
that $ f_2\circ S=T\circ f_1 $.) Therefore if the singular point $p$
of $f$ is a cuspidal edge, then $f$ at $p$ is a front, and
furthermore, they are one of two types of generic singularities of
fronts (the other one is a {\em swallowtail}\/ which is a singular
point $p$ of $f$ satisfying that $f$ at $p$ is
$\mathcal{A}$-equivalent to $(u,v)\mapsto(u, u^2v + 3u^4, 2 uv + 4
u^3)$ at $\zv$). So we state notations and a fundamental property of
singularities of fronts, which are used in the following sections.

 Let $f:(\R^2, \zv)\to(\R^3,\zv)$ be a
front and
$\nu$  a unit normal vector field along $f$, and take $(u,v)$
as a coordinate system of the source. The function
$$
\lambda=\det (f_u,f_v,\nu)
$$
is called the {\em signed area density}, where $f_u=\partial
f/\partial u$ and $f_v=\partial f/\partial v$. A singular point $q$
of $f$ is called {\em non-degenerate}\/ if $d\lambda(q)\ne0$. If $q$
is a non-degenerate singular point of $f$, then the set of singular
points $S(f)$ is a regular curve, which we shall call the
\emph{singular curve} at $q$, and we shall denote by $\gamma$ a
parametrization for this curve. The tangential 1-dimensional vector
space of the singular curve $\gamma$ is called the \emph{singular
direction}. Furthermore, if $q$ is a non-degenerate singular point,
then a non-zero smooth vector field $\eta$ on $(\R^2,\zv)$ such that
$df(\eta)=\zv$ on $S(f)$ is defined. We call $\eta$ a {\em null
vector field} and its direction the \emph{null direction}.  For
details see \cite{front}.

\begin{lemma}
{\rm (\cite[Corollary 2.5, p.735]{suy3}}, {\rm see also
\cite{krsuy})} Let\/ $\zv$ be a singular point of a front\/
$f:(\R^2, \zv)\to(\R^3,\zv)$. Then\/ $\zv$ is a cuspidal edge if and
only if\/ $d\lambda(\eta)\ne0$ at\/ $\zv$. In particular, at a
cuspidal edge, the null direction and the singular direction are
transversal.
\end{lemma}

In this paper we shall use the first and second fundamental forms
 defined in \cite{mn} for surfaces  in $\R^3$ with corank 1 singularities  and given as follows.
Let $q\in \R^2$ be a corank 1 singular point of $f:\R^2\to\R^3$ and
$p=f(q)$. The Euclidean metric $\inner{~}{~}$ of $\R^3$ induces a
pseudometric on $T_{q}\R^2$ given by the {\em first fundamental
form\/} $\firff:T_{q}\R^2 \times T_{q}\R^2\rightarrow \R$ defined by
$\firff(X,Y)=\inner{df_q(X)}{df_q(Y)}$, where $df_q$ is the
differential map of $f$ at $q$. The coefficients of $\firff$ at $q$
are
$$E(q) = \inner{f_u}{f_u}(q),\
F(q) = \inner{f_u}{f_v}(q)\ \text{and}\
G(q) = \inner{f_v}{f_v}(q),$$
and, given $ X=x
\partial_u+y \partial_v\in T_{q}\R^2 $,
then $\firff(X,X)= x^2E(q) + 2xy F(q) + y^2 G(q)$, where $(u,v)$ is
a coordinate system on the source and $\partial_u=(\partial/\partial
u)_q$ and $\partial_v=(\partial/\partial v)_q$.

Let us denote the image of $f$ by $M$ and the tangent line to $M$ at
$p$ by $T_{p}M = \image df_q$.  So there is a plane $N_{p}M$
satisfying $T_p{\R}^3=T_{p}M \oplus N_{p}M$. Consider the orthogonal
projection
$$
\begin{array}{rclrl}
 \bot: T_p\R^3 & \to
N_{p}M\\
 w &\to w^{\bot}
\end{array}\, .$$
The {\em second fundamental form\/} $\secff: T_{q}\R^2 \times
T_{q}\R^2 \rightarrow N_{q}M$ is defined by $\secff(\partial_u,
\partial_u)= f_{uu}^\bot(q)$, $\secff(\partial_u,
\partial_v)= f_{uv}^\bot(q)$ and $\secff(\partial_v,
\partial_v)= f_{vv}^\bot(q)$, and we extend
it in the unique way as a symmetric bilinear map. The {\em second
fundamental form along a normal vector} $\nu \in N_{p}M$ is the
function $\secff_\nu: T_{q}\R^2 \times T_{q}\R^2 \rightarrow \R$
defined by $\secff_\nu(X,Y) = \inner{\secff(X,Y)}{\nu}$ and its
coefficients at $q$ are:
$$
l_\nu(q)=\inner{f_{uu}^\bot(q)}{\nu},\quad
m_\nu(q)=\inner{f_{uv}^\bot(q)}{\nu},\quad
n_\nu(q)=\inner{f_{vv}^\bot(q)}{\nu}.
$$
It is showed in \cite{mn} that $\Delta_{p} = \{\secff(X,X) \, |
\, \firff(X,X)^{1/2} = 1\}$ is a parabola in $N_pM$, which can
degenerate in a half-line, line or a point. ($\Delta_{p}$ is called
the {\em curvature parabola\/} of $M$ at ${p}$.) If $p$ is a
cuspidal edge, $\Delta_{p}$
is a half-line in $N_{p}M$.  For details see \cite{mn}.
%

\section{Normal form of cuspidal edges}
\label{sec:west:type}
In this section we give a normal form of cuspidal edges by
using only
coordinate transformations on the source and isometries on the
target.
These changes of coordinates do preserve the geometry of
the image.

Let $f:(\R^2,\zv)\to(\R^3,\zv)$ be a map-germ and $\zv$ a cuspidal
edge with $f=(f_1,f_2,f_3)$. Let $\nu=(\nu_1,\nu_2,\nu_3)$ be a unit
normal vector field along  $f$ and $(u,v)$  the usual Cartesian
coordinate system of $\R^2$. So $\rank df_{\zv}=1$ and then we may
assume that $f_u(\zv)=\big((f_1)_u(\zv),0,0\big)$, where
$(f_1)_u(\zv)\ne0$,
 by a
rotation of $\R^3$ if necessary.
The map on the source $(\tilde u,\tilde
v)=(f_1(u,v),v)$ is a coordinate transformation. In fact,
$$
\det\pmt{\tilde u_u&\tilde u_v\\ \tilde v_u&\tilde v_v} =
\det\pmt{\tilde u_u&\tilde u_v\\ 0&1}= \tilde
u_u=(f_1)_u\ne0\quad\text{at}\ \zv.$$

By coordinates $(\tilde u,\tilde v)$, $f$ is written  $f(\tilde
u,\tilde v) = (\tilde u,\tilde f_2(\tilde u,\tilde v) ,\tilde
f_3(\tilde u,\tilde v))$ for some functions $\tilde f_2, \tilde
f_3$. Needless to say, $ (\tilde f_2)_u=(\tilde f_3)_u= (\tilde
f_2)_v=(\tilde f_3)_v=0$ at $\zv$.

Since $\zv$ is a cuspidal edge and $\eta=\partial_v$ at $\zv$ is a
null vector field, then $\lambda_v\ne0$. Rewriting $f(u,v) =
(u,f_2(u,v),f_3(u,v))$, we have
$$
\lambda
=
\det\pmt{
1&(f_2)_u&(f_3)_u\\
0&(f_2)_v&(f_3)_v\\
\nu_1&\nu_2&\nu_3}
$$
and therefore $(0,(f_2)_{vv},(f_3)_{vv})\ne \zv$. Since $S(f)$ is a
regular curve, $S(f)$ is transverse to the $v$-axis. Thus $S(f)$
can be parametrized by $(u,g(u))$. Considering the coordinate
transformation on the source
$$
\tilde u=u,\quad \tilde v=v-g(u),
$$
we may assume that $f(u,v)=(u,f_2(u,v),f_3(u,v))$
and $S(f)=\{v=0\}$.

On the other hand,
there exist functions $a_2,a_3,b_2,b_3$ such that
$$
f_i(u,v)=a_i(u)+vb_i(u,v),\quad i=2,3.
$$
Since  $f_v= \zv$ on $\{v=0\}$, it holds that $b_i(u,v)=0$ on
$\{v=0\}$, $i=1,2$. Then by the Malgrange preparation theorem, there
exist functions $\bar b_2,\bar b_3$ such that $b_i(u,v)=v\bar
b_i(u,v)$, $i=1,2$.

Rewriting $\bar b$ to $b$, we may assume that $f$ is of the form
$$
f(u,v)
=\big(u,a_2(u)+v^2b_2(u,v),a_3(u)+v^2b_3(u,v)\big).
$$

By the above arguments, $f_{vv}(\zv)\ne \zv$, that is,
$(b_2,b_3)=((f_2)_{vv},(f_3)_{vv})\ne(0,0)$ at $\zv$.

Now, using  the rotation of $\R^3$ given by the matrix
\begin{equation}
\label{eq:rot}
A_\theta=\pmt{1&\zv\\
\zv&\tilde A_\theta},\quad
\tilde A_\theta=\pmt{
\cos\theta&-\sin\theta\\
\sin\theta&\cos\theta},
\end{equation}
we get $$
\begin{array}{l}
A_\theta f = \big(u, \cos\theta a_2(u)-\sin\theta a_3(u)
+v^2[\cos\theta b_2(u,v)-\sin\theta b_3(u,v)],\\
\hspace{40mm} \sin\theta a_2(u)+\cos\theta a_3(u) +v^2[\sin\theta
b_2(u,v)+\cos\theta b_3(u,v)] \big).
\end{array}
$$
Since $(b_2,b_3)\ne(0,0)$ at $\zv$, there exists some number
$\theta$ such that
\begin{equation}
\label{eq:change}
\cos\theta b_2(\zv)-\sin\theta b_3(\zv)>0
\quad\text{and}\quad
\sin\theta b_2(\zv)+\cos\theta b_3(\zv)=0.
\end{equation}
Setting
$$
\begin{array}{rcl}
\bar a_2(u)&=&\cos\theta a_2(u)-\sin\theta a_3(u),\\
\bar a_3(u)&=&\sin\theta a_2(u)+\cos\theta a_3(u),\\
\bar b_2(u,v)&=&\cos\theta b_2(u,v)-\sin\theta b_3(u,v),\\
\bar b_3(u,v)&=&\sin\theta b_2(u,v)+\cos\theta b_3(u,v),
\end{array}
$$
$f$ is rewritten as
$$
f(u,v) =(u,\bar a_2(u)+v^2\bar b_2(u,v),\bar a_3(u)+v^2\bar b_3(u,v)),
$$
with
$\bar a_2(0)=\bar a_2'(0)=\bar a_3(0)=\bar a_3'(0)=0$,
$\bar b_2(\zv)\ne0$ and $\bar b_3(\zv)=0$,
where $\bar a_2'=d\bar a_2/du$, for example.
We remark that $\bar b_2(\zv)>0$ holds.

Next, using the coordinate transformation on the source
$$
\tilde u=u,\quad
\tilde v=v\sqrt{2\bar b_2(u,v)},
$$
one can rewrite $f$ as
$$
f(\tilde u,\tilde v) =\left(\tilde u,\bar a_2(\tilde u)+\dfrac{v^2}{2}, \bar
a_3(\tilde u)+\tilde v^2\tilde b_3(\tilde u,\tilde v)\right)\,,
$$
for some function $\tilde b_3$ satisfying $\tilde b_3(\zv)=0$.

Rewriting $\tilde u$ as $u$, $\bar a$ as $a$
 and $\tilde b$ as $b$, we may
assume that $f$ is of the form
$$
f(u,v)
=\left(u,a_2(u)+\dfrac{v^2}{2},a_3(u)+v^2b_3(u,v)\right).
$$
Since $b_3(\zv)=0$, there exist functions $a_4(u)$ and $b_4(u,v)$
such that $b_3(u,v)=a_4(u)+vb_4(u,v)$, with  $a_4(0)=0$. We remark
that, as $\zv$ is a cuspidal edge, $d\lambda(\eta)\ne0$. This is
equivalent to $(b_3)_v(\zv)\ne0$. Thus $b_4(\zv)\ne0$.

Hence $f$ can be written as
$$
f(u,v) =\left(u,a_2(u)+\dfrac{v^2}{2},a_3(u)+v^2a_4(u)+v^3b_4(u,v)\right).
$$
Changing the numbering, we get
\begin{equation}\label{eq:west}
f(u,v) =\left(u,a_1(u)+\dfrac{v^2}{2},b_2(u)+v^2b_3(u)+v^3b_4(u,v)\right).
\end{equation}
where $a_1(0)=a_1'(0)=b_2(0)=b_2'(0)=b_3(0)=0$, $b_4(\zv)\ne0$. By
rotations 
$(u,v)\mapsto(-u,-v)$ on $\R^2$ and
$(x,y,z)\mapsto(-x,y,-z)$ on $\R^3$, we may assume that
$b_2''(0)\geq0$.
Summerizing up the above arguments,
we have the following theorem.
\begin{theorem}
Let\/ $f:(\R^2,\zv)\to(\R^3,\zv)$ be a map-germ and\/ $\zv$ a
cuspidal edge. Then there exist a diffeomorphism-germ\/ $\phi:(\R^2,
\zv)\to(\R^2, \zv)$ and an isometry-germ\/ $\Phi:(\R^3,
\zv)\to(\R^3, \zv)$ satisfying that
\begin{equation}
\label{eq:west3}
\begin{array}{l}
\Phi\circ
f\circ\phi (u,v)\\[2mm]
\hspace{10mm}
 =
\displaystyle
\Big(u,
\frac{a_{20}}{2}u^2+\frac{a_{30}}{6}u^3+\frac{1}{2}v^2,
\frac{b_{20}}{2}u^2+\frac{b_{30}}{6}u^3+\frac{b_{12}}{2}uv^2
+\frac{b_{03}}{6}v^3\Big)+h(u,v),\\[4mm]
\hspace{105mm}
(b_{03}\ne0,\ b_{20}\geq0),
\end{array}
\end{equation}
where
$$h(u,v)= \big(
0,\,
u^4h_1(u),
u^4h_2(u)+u^2v^2h_3(u)+uv^3h_4(u)+v^4h_5(u,v)
\big),
$$
with\/ $h_1(u),h_2(u),h_3(u),h_4(u),h_5(u,v)$  smooth functions.
\end{theorem}
We call this parametrization the {\em normal form}\/ of
cuspidal edges. One can easily verify that all coefficients of
\eqref{eq:west3} are uniquely determined, since the rotation
\eqref{eq:change} means that $\eta\eta f(0)=(0,1,0)$, where $\eta f$
means the directional derivative $df(\eta)$. This unique expansion
of a cuspidal edge implies that the above coefficients can be
considered as geometric invariants of the cuspidal edge. It means
that any cuspidal edge has this form using only coordinate changes
on the source and isometries of $\R^3$.

We shall deal with the six geometric invariants of cuspidal
edges given by the formula \eqref{eq:west3} in the following
sections.

\section{Singular curvature, normal curvature and umbilic curvature}
In this section, we review the singular
curvature $\kappa_s$,
the limiting normal curvature $\kappa_n$
(\cite{front}) and the umbilic curvature
$\kappa_u$ (\cite{mn}).
We show $\kappa_n=\kappa_u$ and
compute the curvature $\kappa$ and the torsion $\tau$, as well as
$\kappa_s$ and $\kappa_n$, of the cuspidal edge given by
\eqref{eq:west3}. As an immediate consequence
we obtain an expression
relating the singular and limiting normal
curvatures with the curvature of
the cuspidal curve as a space curve.

Let $f:(\R^2,\zv) \to (\R^{3},\zv)$ be a map-germ, and $\zv$ a
cuspidal edge. Let $\gamma(t)$ be the singular curve, $\hat{\gamma}
= f \circ \gamma$ and choose the null vector $\eta(t)$ such that
$(\gamma'(t),\eta(t))$ is a positively oriented frame field along
$\gamma$. The \emph{singular curvature}\/ $\kappa_s$ and the {\em
limiting normal curvature}\/ $\kappa_{\nu}$ at $t$ are the functions
(\cite{front})
\begin{equation}\label{eq:def:ks}
\begin{array}{rcl}
\kappa_s (t) &=&
\displaystyle
 \sgn (d\lambda(\eta))\, \frac{\det
(\hat{\gamma}'(t), \hat{\gamma}''(t),
\nu(\gamma(t)))}{|\hat{\gamma}'(t)|^3} \, = \, \sgn (d\lambda(\eta))\,
\frac{\langle \hat{\gamma}''(t) ,
n(t)\rangle}{|\hat{\gamma}'(t)|^2}\,,\\[4mm]
\kappa_{\nu} (t) &=& \displaystyle \frac{
\inner{\hat\gamma''(t)}{\nu(\gamma(t))} }{|\hat{\gamma}'(t)|^2}\,,\\[4mm]
n(t) &=&
\nu(\gamma(t)) \times
\dfrac{\hat{\gamma}'(t)}{|\hat{\gamma}'(t)|},
\end{array}
\end{equation}
where $\times$ denotes the vector product in $\R^3$. Then
$\kappa_s(t)$ can be considered as the limiting geodesic curvature
of curves with the singular curve on their right-hand sides. 
The definitions given in (\ref{eq:def:ks}) do not depend on the
parametrization for the singular curve, nor the orientation of
$\R^2$. Furthermore, $\kappa_s$ does not depend on the choice of
$\nu$, and $\kappa_{\nu}$ depends on the choice of $\nu$. For more
details see \cite{front}. We consider the absolute value of
$\kappa_{\nu}$ and set $\kappa_n=|\kappa_{\nu}|$. We call $\kappa_n$
the {\em absolute normal curvature} or, shortly,  the {\em normal
curvature}
 of the cuspidal edge.

The umbilic curvature $\kappa_u$ is a function defined in \cite{mn}
for corank 1 singular points of surfaces in $\R^3$, unless for
Whitney umbrellas (i.e., surfaces image of any map germ $(\R^2, \zv)
\rightarrow (\R^3, \zv)$ which is $\mathcal{A}$-equivalent to
$(x,y^2,xy)$), and so $\kappa_u$ is well defined for the cuspidal
edge $f$ at $\hat{\gamma}(t)$. Its definition is given in terms of
the first and second fundamental forms of $M = \image f$ defined in
Section \ref{sec:prelim}.

Under the above setting,
let $\alpha: \R \rightarrow N_pM$ be a parametrization for
$\Delta_p$, where $p = \hat{\gamma}(t)$. 
Since $\Delta_p$ is a half-line,
$|\alpha(s) \times \alpha'(s)|/|\alpha'(s)|$
does not depend on the parametrization
$\alpha(s)$ for $\Delta_p$, 
nor on the value $s$ satisfying $\alpha'(s) \neq 0$.
Set $\kappa_u(t)=|\alpha(s) \times \alpha'(s)|/|\alpha'(s)|$.
Since $N_pM$ is a normal plane of $\hat{\gamma}'(t)$,
\begin{equation}\label{eq:def:ku}
\kappa_u(t)=
\frac{|\alpha(s) \times \alpha'(s)|}{|\alpha'(s)|}
=
\left|
\inner{\frac{|\alpha(s) \times \alpha'(s)|}{|\alpha'(s)|}}
{\frac{\hat{\gamma}'(t)}{|\hat{\gamma}'(t)|}}
\right|
=
\frac{
|\det (\alpha(s), \alpha'(s), \hat{\gamma}'(t))|}{|\alpha'(s)
\times \hat{\gamma}'(t)|}
\end{equation}
holds for any $s$ such that $\alpha'(s) \neq 0$.
Notice that $\kappa_u(t)$ is the distance
between $p$ and the line $\ell$ containing $\Delta_p$.

For later computation, it is convenient to take an \emph{adapted
pair of vector fields} and an \emph{adapted coordinate system}. If a
singular point of the map-germ $f$ is a cuspidal edge, then $S(f)$
is a regular curve on the source and the null vector field is
transverse to $S(f)$. Thus we can take a pair of vector fields and a
coordinate system as follows: A pair of vector field $(\xi, \eta)$
on $(\R^2,\zv)$ is called {\em adapted\/} if it satisfies:
\begin{enumerate}
\item $\xi$ is tangent to $S(f)$ on $S(f)$,
\item $\eta$ is a null vector on $S(f)$, and
\item $(\xi,\eta)$ is positively oriented.
\end{enumerate}
A coordinate system $(u,v)$ on $(\R^2,\zv)$ is called {\em
adapted\/} if it satisfies:
\begin{enumerate}
\item the $u$-axis is the singular curve,
\item $\partial_v$ gives a null vector field on the $u$-axis,
and
\item there are no singular points except the $u$-axis.
\end{enumerate}
We remark that the coordinate system $(u,v)$ in the formula
\eqref{eq:west3} is adapted.
Condition (1) for adapted vector field $(\xi,\eta)$ is characterized by
$\xi\lambda=0$ on $S(f)$, where $\lambda$ is
the signed area density,
and Condition (2) is characterized by
$\eta f=0$ on $S(f)$.
Formulas for coefficients in \eqref{eq:west}
by using adapted coordinate systems
are stated in the following sections.
In \cite{msuy}, we also define adapted coordinate system. In that
definition, a condition $|f_u|=1$ is imposed in addition to the
above, but we do not assume it here.
\begin{remark}
\label{rem:adapted} If $(\xi, \eta)$ is an adapted pair of vector
fields, then $\xi\eta f= \zv$ holds on $S(f)$, since $\eta f=\zv$ on
$S(f)$. Furthermore, $\{\xi f, \eta\eta f,\nu\}$ is linearly
independent, since $\det(\xi f, \eta\eta f,\nu)=\eta\lambda\ne0$ at
$\zv$. For the same reason, if $(u,v)$ is an adapted coordinate
system, then $f_{uv}= \zv$ holds on $S(f)$ and $\{f_u,f_{vv},\nu\}$
is linearly independent.
\end{remark}

Taking an adapted coordinate system, it holds that
\begin{equation}\label{eq:ks:adap}
\kappa_s(u,0)= \sgn(\lambda_v) \frac{\det (f_u,
f_{uu},\nu)}{|f_u|^3}(u,0).
\end{equation}
See \cite{front} for details.

For an adapted pair of vector fields $(\xi,\eta)$ on $(\R^2,\zv)$,
it
can be easily seen that
\begin{equation}
\kappa_s(u,v)= \sgn(\eta\lambda) \frac{\det (\xi f, \xi\xi
f,\nu)}{|\xi f|^3}(u,v), \quad (u,v)\in S(f),
\end{equation}
and
\begin{equation}
\label{eq:kn} \kappa_n(u,v) = \frac{|\det(\xi f, \eta\eta f, \xi\xi
f)|} {|\xi f|^2|\xi f\times \eta\eta f|}(u,v) \Bigg(=
\frac{\sgn(\eta\lambda\,\inner{\nu}{\xi\xi f}) \det(\xi f, \eta\eta
f, \xi\xi f)} {|\xi f|^2|\xi f\times \eta\eta f|}(u,v)\Bigg),
\end{equation}
where $(u,v)\in S(f)$.

\begin{lemma}
The formula\/ $\eqref{eq:kn}$ of\/ $\kappa_n(u,v)$ does not depend
on the choice of pairs of adapted vector fields.
\end{lemma}
\begin{proof}
Let us take another pair of adapted vector fields
$(\tilde \xi,\tilde \eta)$ such that
\begin{equation}
\label{eq:newvf}
\tilde \xi=a\xi+b\eta,\quad
\tilde \eta=c\xi+d\eta,
\end{equation}
where $a,b,c,d$ are smooth functions of $(u,v)$ satisfying
$ad-bc\ne0$, and on $S(f)$, satisfying $b=c=0$. Moreover, as $(\tilde \xi,\tilde
\eta)$ is positively oriented on $S(f)$, $ad>0$ holds on
$S(f)$.
Then we have
\begin{equation}
\label{eq:newvf1}
\begin{array}{rcll}
\tilde \xi f
&=&
a \xi f+b\eta f
\\
&=&
a \xi f&(\text{on}\ S(f)),
\\
\tilde \xi \tilde \xi f
&=&
a(\xi a \xi f + a \xi \xi f + \xi b \eta f + b \xi \eta f)\\
&&\hspace{30mm}
+b(\eta a \xi f + a \eta \xi f + \eta b \eta f + b \eta \eta f)\\
&=&
a\xi a\xi f+a^2 \xi\xi f&(\text{on}\ S(f)),
\\
\tilde \eta f
&=&
c\xi f+d\eta f
\\
\tilde \eta\tilde \eta f
&=&
c(\xi c \xi f+c  \xi\xi f+\xi d\ \eta f+d \xi\eta f)\\
&&\hspace{30mm}
+
d(\eta c \xi f+c  \eta\xi f+\eta d\ \eta f+d \eta\eta f)\\
&=&d\eta c\xi f+d^2\eta\eta f &(\text{on}\ S(f)).
\end{array}
\end{equation}
Then it holds that
%
$$
\frac{\det(\tilde\xi f,\tilde\eta\tilde\eta f, \tilde\xi\tilde\xi
f)} {|\tilde\xi f|^2|\tilde\xi f\times \tilde\eta\tilde\eta f|} =
\frac{a^3d^2\det(\xi f, \eta\eta f, \xi\xi f)} {|a^3||d^2||\xi
f|^2|\xi f\times \eta\eta f|} =\sgn(a)\frac{\det(\xi f, \eta\eta f,
\xi\xi f)} {|\xi f|^2|\xi f\times \eta\eta f|},
$$
and the lemma follows.
\end{proof}

Since the formula $\eqref{eq:kn}$ does not depend on the choice of pairs of
adapted vector fields, one can choose $\xi=\partial_u$,
$\eta=\partial_v$ and an adapted coordinate system, getting
\begin{equation} \label{eq:def:ku:vector}
{\kappa}_n(u,0) = \frac{1}{E} \frac{|\det (\,f_{u}\, ,f_{vv}\,
,f_{uu})|}{|f_{u}
 \times f_{vv}|} \,(u,0)\ .
\end{equation}

We can state properties of the umbilic curvature and singular
curvature in terms of the second fundamental form and also relate
the umbilic and normal curvatures.

\begin{theorem}\label{thm:ku}
Let $f: (\R^2, q) \to (\R^3, p)$ be a map-germ, \/ $q$ a cuspidal
edge and $\nu$ a unit normal vector field along $f$. Then the
following hold.
\begin{enumerate}
\renewcommand{\labelenumi}{{\rm (\alph{enumi})}}
\item
\label{item:a} $\nu(q)$ is orthogonal to the line $\ell$ which
contains $\Delta_{p}$\,. Therefore, 
it holds that $\displaystyle \kappa_u(q) =
\frac{|{\secff}_\nu (X,X)|}{\firff(X,X)}$, for any $X \in T_q\R^2$.

\item \label{item:normalumb} $\kappa_u(q)=\kappa_n(q)$.

\item
\label{item:c} $\kappa_s(q) = 0 $ if and only if \,$\secff (X,X)$ is
parallel to $\nu$ at $p$, where $X$ is a non-zero tangent vector to $S(f)$ at
$q$.

\item
\label{item:d} $\kappa_u(q) = \kappa_s(q) = 0$ if and only if
\,$\secff(X,X) = \zv$, where $X$ is a non-zero tangent vector to $S(f)$ at
$q$.
\end{enumerate}
\end{theorem}
\begin{proof}
Let $M$ be the image of $f$. It was shown in \cite[Lemma 2.1]{mn}
that
  the second fundamental form  does not depend
on the choice of local system of coordinates on $(\R^2, \zv)$. So we
can take an
  adapted coordinate system $(u,v)$.

 Writing $\gamma(t) = (u(t),0)$ and $\hat{\gamma}(t) = f \circ
\gamma (t)$,
 then  $\hat{\gamma}'(t) = u'(t)f_u(u(t),0)$.
 Since $f_v (u(t),0) =\zv$ then
$X=x \partial_u + y \partial_v$ is a unit vector in
$T_{\gamma(t)}\R^2$ (with relation the pseudometric given by
$\firff$) if and only if $x= \pm {1/\sqrt{E(u(t),0)}}$. As
$f_{uv}(u(t),0) = \zv$, then a parametrization for
$\Delta_{\hat{\gamma}(t)}$ at $\hat{\gamma}(t)$ is
$$
\alpha (s) = \frac{1}{E(u(t),0)} f_{uu}^\bot(u(t),0)  +
s^2f_{vv}^\bot(u(t),0),
$$ and so
$\alpha' (s) =
 2sf_{vv}^\bot\,(u(t),0).$

 Then, to conclude (a) it
is enough to verify that $\nu(q)$ is orthogonal to $f_{vv}^\bot(q)$.
 On $S(f) = \{(u,v); v =0\}$ it also holds that $f_v =  \zv$
and by Remark \ref{rem:adapted}, $f_{vv} \neq \zv$ holds.
Therefore we can write $f_v = v h$,
where $h(u,v) \neq \zv$ on $(\R^2,\zv)$, which implies that
 \begin{equation}\label{eq:nu:Sf}
 \displaystyle \nu \, = \, \varepsilon \,
\frac{f_{u}\times h}{|f_{u}\times h|} \, = \, \varepsilon \,
\frac{f_{u}\times f_{vv}}{|f_{u}\times f_{vv}|} \, = \,\varepsilon
\, \frac{f_{u}\times f_{vv}^\bot}{|f_{u}\times f_{vv}|}
\end{equation}
 on
the singular set, where $\varepsilon=1$ or $-1$, and  therefore
$\nu$ is orthogonal to $f_{vv}^\bot$ on $S(f)$, as we claimed.
So, it follows from Remark 3.10(3) of \cite{mn} that
$$\displaystyle \kappa_u(q) = \frac{|{\secff}_\nu
(X,X)|}{\firff(X,X)}$$ holds for any $X \in T_q\R^2$.

%

 Denoting $\kappa_u(t)$ by $\kappa_u(u,0)$, it holds from
(\ref{eq:def:ku}) that
$$
\kappa_u(u,0) = \frac{ |\det (\frac{1}{E} f_{uu}^\bot +
s^2f_{vv}^\bot,\
       2sf_{vv}^\bot,\
       f_u)|}
{|2sf_{u}^\bot \times f_{vv}|}(u,0) = \frac{|\det (
f_{uu},\
       2sf_{vv},\
       f_u)|}
{E|2sf_{u}\times f_{vv}|}(u,0).
$$
Therefore, from \eqref{eq:def:ku:vector}, we get that
$$
\kappa_u(u,0) = {\kappa}_n(u,0),
$$
concluding (b).

Hence,
\begin{equation} \label{eq:ku-normal}
\begin{array}{l}
\displaystyle \kappa_u(u,0)
=  \frac{|\det (\,f_{u}\,
,f_{vv}\, ,f_{uu})|} {E|f_u\times f_{vv}|} \ = \ \frac{1}{E} \left|
\inner{ \frac{f_{u}\times f_{vv}}{|f_u\times f_{vv}|}
}{f_{uu}}\right| =
\frac{|\inner{ \nu }{f_{uu}}|}{E}\, =
 \,  \frac{|\inner{ \nu }{f_{uu}^\bot}|}{E} ,
\end{array}
\end{equation}
at $q=(u,0)$.  

 Now, consider the orthonormal frame
$\{\nu(q), \nu(q) \times f_u(q)/|f_u(q)|\}$ for $N_pM$. Noticing
that $\det(f_u, f_{uu}, \nu) = \det (f_u, f_{uu}^\bot, \nu)$ at $q$,
then, from (\ref{eq:ks:adap}) and (\ref{eq:ku-normal}), it holds
that
$$\frac{1}{E}\secff (\partial_u, \partial_u) \,=\, \frac{1}{E}f_{uu}^\bot \, = \,  \kappa_u \, \nu \, + \,
\sgn(\lambda_v) \kappa_s \ \nu \times \frac{f_u}{|f_u|}\,,$$
at $q$, which implies that $\left|\secff (\partial_u,
\partial_u)/{E(q)}\right|^2 = \kappa_u^2(q) + \kappa_s^2(q)$,
and consequently we conclude (c) and (d) of the theorem.
\end{proof}
\medskip

A usual approach for getting information about the geometry of
surfaces is analyzing their generic contacts with planes and
spheres. Such contacts are measured by composing the implicit
equation of the plane or sphere with the parametrisation of the
surface, and seeing what types of singularities arise. Then we label
the contact according with the type of singularity.  In \cite{mn}
J. J. Nu\~no-Ballesteros and the first author deal with such study
for surfaces in $\R^3$ with corank 1 singularities.  We recall that
a singular point of a function is said to be of type $\Sigma^{2,2}$
if all of the partial derivatives of the function up to order 2 at the
singular point are equal to zero.  With the conditions of 
Theorem \ref{thm:ku} and by \cite{mn}, 
it follows that: (a) If $\kappa_n(q) = 0$,
then the plane at $p$ orthogonal to $\nu(q)$ is the only plane in
$\R^3$ having contact of type $\Sigma^{2,2}$ with $f$.
(b) If $\kappa_n(q) \neq 0$, then 
the sphere with
center at 
$$u = p +  \varepsilon
\frac{1}{\kappa_n(q)} \,\nu(q)$$
is the only sphere in $\R^3$ having
contact of type $\Sigma^{2,2}$ with $f$, where
$\varepsilon = \sgn(\secff_\nu(X,X))$, for any unit vector $X \in
T_q\R^2$.

\medskip

When $f$ is of normal form,
the relations between the singular curvature,
the limiting normal curvature
(so, the umbilic curvature)
and curvature and torsion of the space curve
$f|_{S(f)}$ are given in the following result.
\begin{theorem}
\label{thm:rel1} Let $f(u,v)$ be a map-germ of the form
\eqref{eq:west3}. For the space curve $f|_{S(f)}$ at the origin, it
holds that
$$
\begin{array}{l}
\kappa_s=a_{20},\quad
\kappa_s'=a_{30}+b_{12} b_{20},\quad
\kappa_n=\kappa_u=b_{20},\quad
\kappa_n'=b_{30}-a_{20} b_{12},
\\
\kappa=\sqrt{a_{20}^2+b_{20}^2},\quad
\kappa'=\dfrac{a_{20} a_{30} + b_{20} b_{30}}
{\sqrt{a_{20}^2 + b_{20}^2}},\quad
\tau=\dfrac{a_{20}a_{30}-b_{20}a_{30}}{a_{20}^2+b_{20}^2}.
\end{array}
$$
\end{theorem}
\begin{proof}
The curvature and the torsion of the curve $\hat{\gamma} = f \circ
\gamma$ are calculated as  usual  and we shall omit that here. Using
the parametrization $f(u,v)$ given by  \eqref{eq:west}, we have:
$$ f_u = \big(1, a_1', b_2' + v^2 b'_3+ v^3 (b_4)_u\big) \text{\ and\ }
f_v = \big(0, v, 2vb_3 + 3v^2 b_4+ v^3 (b_4)_v\big),
$$
 where $a_1'= {da_1}/{du}$ and $b_i'= {db_i}/{du}$  for $i=2,3$.
  This implies that the $u$-axis is the singular
 curve and
 the $v$-direction is the null direction. So $(u,v)$ is an adapted
 system of coordinates. Since $f_u(\zv)=(1,0,0)$,
 $f_{uu}(\zv) = (0, a_1''(0), b_2''(0))$,  $f_{vv}(\zv) = (0, 1,0)$,
 a unit normal vector at $\zv$ is $\nu = (0,0,1)$ and the signed area
 density satisfies $\lambda_v(\zv) = 1$. Then,  using \eqref{eq:ks:adap} and \eqref{eq:ku-normal}, we get
 $\kappa_s(\zv) = a_1''(0) = a_{20}$ and $\kappa_n(\zv) =b_2''(0) =
 b_{20}.$
\end{proof}

As a consequence, we have the following corollary.
\begin{corollary}\label{cor:rel:ku:ks}
Let $f: (\R^2, \zv) \rightarrow (\R^3, \zv)$ be a map-germ, and
$\zv$ a cuspidal edge, $\gamma(t)$ a parametrization of $S(f)$ and
$\hat\gamma(t)=f\circ\gamma(t)$. Let $\kappa(t)$ be the curvature of
$\hat\gamma(t)$ as a curve in $\R^3$, $\kappa_s(t)$ its singular
curvature and $\kappa_n(t)$ its limiting normal curvature. Then
$$
\kappa(t)^2=\kappa_s(t)^2+\kappa_n(t)^2.
$$
\end{corollary}


\section{Other geometric invariants up to order three}
\label{sec:invar} Comparing \eqref{eq:west3} and Theorem
\ref{thm:rel1}, there are three other independent invariants of
cuspidal edges up to order three.

\subsection{Cuspidal curvature}

The {\em
cuspidal curvature\/} $\kappa_c$ for cuspidal edges is defined in
\cite{msuy} as
$$
\kappa_c (u,v)=\frac{|\xi f|^{3/2} \det(\xi f,\ \eta\eta f,\
\eta\eta\eta f)} {|\xi f\times\eta\eta f|^{5/2}}\, (u,v)\, , \ \
(u,v) \in S(f),
$$
where $(\xi,\eta)$ is an adapted pair of vector fields on
$(\R^2,\zv)$. If $f(u,v)$ is a map-germ of the form
\eqref{eq:west3}, then it holds that
$$
\kappa_c(0,0)=b_{03}.
$$
See \cite{msuy} for detailed description and
geometric meanings of it.

\medskip

\subsection{Cusp-directional torsion}
Let $f=(f_1,f_2,f_3):(\R^2,\zv)\to(\R^3,\zv)$ be 
a map-germ, $\zv$ a
cuspidal edge, and $\gamma(t)$ a parametrization of $S(f)$. Take a
pair of adapted vector fields $(\xi,\eta)$  on $(\R^2, \zv)$.
We define the \emph{cusp-directional torsion} on singular points
consisting of cuspidal edges as follows:
\begin{equation}
\label{eq:wobs}
\kappa_t(u,v)=
 \frac{\det(\xi f,\,\eta\eta f,\,\xi\eta\eta f)} {|\xi
f\times\eta\eta f|^2}(u,v) - \frac{\det(\xi f,\,\eta\eta f,\,\xi\xi
f) \inner{\xi f}{\eta\eta f}} {|\xi f|^2|\xi f\times\eta\eta f|^2}
(u,v),\ (u,v)\in S(f).
\end{equation}
By Remark \ref{rem:adapted}, the denominator of $\kappa_t(u,v)$ does
not vanish and therefore $\kappa_t$ is a bounded function on
cuspidal edges. The following proposition shows that the cusp-directional
torsion is well defined.
\begin{proposition}
\label{prop:inv}
The definition of cusp-directional torsion does not
depend on the choice of the pair\/ $(\xi, \eta)$ of adapted vector
fields on\/ $(\R^2, \zv)$.
\end{proposition}
\begin{proof}
 Define a new pair $(\tilde{\xi}, \tilde{\eta})$ of adapted vector
fields as in \eqref{eq:newvf}. By \eqref{eq:newvf1}, we have
$$
\tilde\xi\tilde\eta\tilde\eta f
=
x_1\xi f+x_2\eta\eta f
+ad\,\eta c\,\xi\xi f+ad^2 \xi\eta\eta f
$$
holds on $S(f)$, where $x_1,x_2$ are some functions. Thus again by
\eqref{eq:newvf1}, it holds that
\begin{align*}
&\dfrac{\det(\tilde\xi f,\,\tilde\eta\tilde\eta f,\,
\tilde\xi\tilde\eta\tilde\eta f)}
{|\tilde\xi f\times\tilde\eta\tilde\eta f|^2}
-
\dfrac{\det(\tilde\xi f,\,\tilde\eta\tilde\eta f,\,
\tilde\xi\tilde\xi f)
\inner{\tilde\xi f}{\tilde\eta\tilde\eta f}}
{|\tilde\xi f|^2|\tilde\xi f\times\tilde\eta\tilde\eta f|^2}\\[4mm]
=&
\dfrac{\det(\xi f,\,\eta\eta f,\,\eta c\,\xi\xi f+d\xi\eta\eta f)}
{d|\xi f\times \eta\eta f|^2}
-
\dfrac{
\inner{\xi f}{\eta c\,\xi f+d\eta\eta f}
|\xi f,\ \eta\eta f,\,\xi\xi f|}
{d|\xi f|^2|\xi f\times \eta\eta f|^2}\\[4mm]
=&
\frac{\det(\xi f,\,\eta\eta f,\,\xi\eta\eta f)}
{|\xi f\times\eta\eta f|^2}
-
\frac{\det(\xi f,\,\eta\eta f,\,\xi\xi f)
\inner{\xi f}{\eta\eta f}}
{|\xi f|^2|\xi f\times\eta\eta f|^2}.
\end{align*}
Thus the proposition follows.
\end{proof}

In the following equation we give the expression of the cusp-directional
torsion for an adapted coordinate system:
$$
\kappa_t(u,0)=  \frac{\det(f_u,\,f_{vv},\,f_{uvv})} {|f_{u}\times
f_{vv}|^2}(u,v) - \frac{\det(f_u,\,f_{vv},\,f_{uu})
\inner{f_u}{f_{vv}}} {|f_u|^2|f_u\times f_v|^2} (u,0).
$$
Moreover, if $(u,v)$ satisfies $\inner{f_u}{f_{vv}}(u,0)=0$, then we
have the following simple expression:
$$
\kappa_t(u,0)= \frac{\det(f_u,\, f_{vv},\, f_{uvv})} {|f_u \times
f_{vv}|^{2}}\, (u,0)\,.
$$

\medskip

The next result gives $\kappa_t(\zv)$  for the normal form
of cuspidal edges and we omit its proof, as it is just a
straightforward calculation.

\begin{proposition}
Let $f$ be a map-germ of the form \eqref{eq:west3}.
Then
$
\kappa_t(\zv)= b_{12}
$
holds.
\end{proposition}

 Let us state the geometric meaning of the invariant
$\kappa_t$.
\begin{proposition}
\label{thm:admi} Let $f:(\R^2,\zv)\to(\R^3,\zv)$ be a map-germ,
$\zv$ a cuspidal edge, $\nu$ a unit normal vector field along $f$,
$\gamma(t)$ a parametrization of $S(f)$ and $q \in S(f)$. If
$\operatorname{pr}_{\vec{v}}\circ f$ is locally a bijection, then\/
$\kappa_t=0$ on $S(f)$ near $q$. Here, $\operatorname{pr}_{\vec{v}}$
is the orthogonal projection to the orthogonal plane to
$\vec{v}=\image df_{\zv}(T_{\zv}\R^2)\times \nu(\zv)$.
\end{proposition}
\begin{proof}
We may assume that $f$ is given by \eqref{eq:west}. Then
$$
\operatorname{pr}_{\nu(\zv)}\circ f(u,v)
=\big(u,a_2(u)+v^2a_3(u)+v^3b(u,v)\big).
$$
For a sufficiently small $u_0$,
$(u_0,a_2(u_0)+a_3(u_0)v^2+v^3b(u_0,v))$ is located on the line
$u=u_0$. If $u_0\ne u_1$ then these lines do not have a crossing.
Hence $\operatorname{pr}_{\vec{v}}\circ f$ is locally bijective if
and only if $g_{u}(v)=a_2(u)+a_3(u)v^2+v^3b(u,v)$ is monotone, for
any sufficiently small $u$. Since
$$g'_u(0)=0,\ g''_u(0)=2a_3(u)\ \text{and}\ g'''_u(0)=6b(u,0),$$
a necessary condition that $\operatorname{pr}_{\nu(\zv)}\circ f$ is
locally bijective is $a_3(u)=0$, for any sufficiently small $u$. Since
$a_3(0)=0$ and $a_3'(0)=\kappa_t(\zv)$, this proves the assertion.
\end{proof}

Let $f:(\R^2,\zv)\to(\R^3,\zv)$ be a map-germ, $\zv$ a cuspidal edge
and $M$ the image of $f$. Then the slice locus $M\cap N_{\zv}M$ is a
cusp. When a pair of adapted vector field $(\xi,\eta)$ satisfies
$\inner{\xi f}{\eta\eta f}\equiv0$ on $S(f)$, then $\eta\eta f\in
N_{\zv}M$ points to the direction where the cusp comes in
$N_{\zv}M$. We call this direction of the {\em cusp-direction}.
Proposition \ref{thm:admi} implies that the cusp-directional torsion
measures the rotation of cusp-direction along the singular curve of
the cuspidal edge. This is a reason that we call $\kappa_t$
cusp-directional torsion. Remark that \eqref{eq:wobs} is
well-defined for non-degenerate singularities whose null direction
is transverse to the singular direction (For example, the cuspidal
cross cap $(u,v)\mapsto(u,v^2,uv^3)$). In appendix, a global
property of $\kappa_t$ is discussed.
\medskip

\subsection{Edge inflectional curvature}
The invariants introduced in the previous sections for a map-germ
given by (\ref{eq:west3}) suggest us that there is one more
geometric invariant which should tell us about $a_{30}$ or $b_{30}$.
Let $(\xi, \eta)$ be a pair of adapted vector fields on $(\R^2,
\zv)$. We define the \emph{edge inflectional curvature}\/ as
follows:
$$
\kappa_i(u,v)= \frac{\det(\xi f,\ \eta\eta f,\ \xi\xi\xi f)} {|\xi
f|^3|\xi f\times\eta\eta f|}\,(u,v) - 3\frac{\inner{\xi f}{\xi\xi f}
\det(\xi f,\ \eta\eta f,\ \xi\xi f)} {|\xi f|^5|\xi f\times\eta\eta
f|}\,(u,v),
$$
$(u,v) \in S(f)$. If where $\xi$ is chosen satisfying $|\xi
f|=1$ on $S(f)$, then we have
$$
\kappa_i(u,v)=\frac{\det(\xi f,\ \eta\eta f,\ \xi\xi\xi f)} {|\xi
f\times\eta\eta f|}\,(u,v).
$$
\begin{proposition} The function  $\kappa_i$
does not depend on the choice of the pair $(\xi,\eta)$
of adapted vector fields.
\end{proposition}
\begin{proof}
Define a new pair $(\tilde{\xi}, \tilde{\eta})$ of adapted vector
fields as in \eqref{eq:newvf}. By \eqref{eq:newvf1}, we have
$$
\tilde\xi\tilde\xi\tilde\xi f
=
a\big((\xi a)^2+a\xi\xi a\big)\xi f
+3a^2\xi a\xi\xi f+a^3 \xi\xi\xi f
$$
holds on $S(f)$.
Thus again by \eqref{eq:newvf1},
we see that
\begin{align*}
&\frac{\det(\tilde\xi f,\ \tilde\eta\tilde\eta f,\
\tilde\xi\tilde\xi\tilde\xi f)}
{|\tilde\xi f|^3|\tilde\xi f\times\tilde\eta\tilde\eta f|}
-
3\frac{\inner{\tilde\xi f}{\tilde\xi\tilde\xi f}
\det(\tilde\xi f,\ \tilde\eta\tilde\eta f,\ \tilde\xi\tilde\xi f)}
{|\tilde\xi f|^5|\tilde\xi f\times\tilde\eta\tilde\eta f|}\\[4mm]
=&
\dfrac{3\xi a\det(\xi f,\ \eta\eta f,\ \xi\xi f)}
{a|\xi f|^3|\xi f\times\eta\eta f|}
+
\frac{\det(\xi f,\ \eta\eta f,\ \xi\xi\xi f)}
{|\xi f|^3|\xi f\times\eta\eta f|}\\[4mm]
&\hspace{30mm}
-
\dfrac{3\big(\xi a\inner{\xi f}{\xi f}+a\inner{\xi f}{\xi\xi f}\big)
\det(\xi f,\ \eta\eta f,\ \xi\xi f)}
{a|\xi f|^5|\xi f\times \eta\eta f|}\\[4mm]
=&
\frac{\det(\xi f,\ \eta\eta f,\ \xi\xi\xi f)}
{|\xi f|^3|\xi f\times\eta\eta f|}
-
3\frac{\inner{\xi f}{\xi\xi f}
\det(\xi f,\ \eta\eta f,\ \xi\xi f)}
{|\xi f|^5|\xi f\times\eta\eta f|}
\end{align*}
holds on $S(f)$ as we claimed.
\end{proof}
\medskip

The expression for $\kappa_i$ at an adapted coordinate system is the
following one:
$$
\kappa_i (u,0) =  \frac{\det(f_u,\ f_{vv},\ f_{uuu})}
{|f_u|^3|f_u\times f_{vv}|}\,(u,0) - 3\frac{\inner{f_u}{f_{uu}}
\det(f_u,\ f_{vv},\ f_{uu})} {|f_u|^5|f_u\times f_{vv}|}\,(u,0).
$$
The next result gives the edge inflectional curvature at $\zv$ for
the normal form of cuspidal edges. Its proof is just a
straightforward calculation, and so we omit it here.

\begin{proposition}
Let $f$ be a map-germ given by \eqref{eq:west3}. Then
$
\kappa_i(\zv)= b_{30}
$
holds.
\end{proposition}

Let us consider the geometric meaning of the invariant $\kappa_i$.
Since $\eta f=0$, $\eta\eta f$ points in the direction which the
cusp comes, and $\det(\xi f,\xi\xi\xi f,\eta\eta f) = \inner{\xi
f\times \xi\xi\xi f}{\eta\eta f}$. So, the invariant $\kappa_i$
measures
the difference between the vector $\xi f\times \xi\xi\xi f$ and the
cusp direction $\eta\eta f$. Here, $\xi f\times \xi\xi\xi
f=\hat\gamma'(s)\times\hat\gamma'''(s)$, where $s$ is the arc-length
of $\hat\gamma$. Let $\kappa$ and $\tau$ be the curvature and
torsion of $\hat{\gamma}$, and assume $\kappa>0$. Consider the
Frenet frame $\{\vec{t},\vec{n},\vec{b}\}$ and assume that
$\hat{\gamma}'\times \hat{\gamma}'''$ is  constant. Then, by the
Frenet-Serret formulas, it holds that
$$
\kappa^2\tau\vec{t} + (-2\kappa'\tau-\kappa\tau')\vec{n} +
(\kappa''-\kappa\tau^2) \vec{b}\equiv\zv.
$$
Since $\kappa>0$, we have $\tau\equiv0$ and $\kappa''\equiv0$. This
means that $\hat{\gamma}$ is a plane curve and a clothoid, or a
circle.


\section{Geometric invariants up to order three}\label{sec:main:th}
Let $f$ be a map-germ given by \eqref{eq:west3}. Then it holds that
$$
\kappa=\sqrt{a_{20}^2+b_{20}^2},\
\tau=\dfrac{a_{20}b_{30}-b_{20}a_{30}}{a_{20}^2+b_{20}^2},\
\kappa_s=a_{20},\
\kappa_n=b_{20},\
\kappa_c=b_{03},\
\kappa_t=b_{12},\
\kappa_i=b_{30}
$$
at $\zv$. We see that $\kappa$ is written in terms of
$a_{20}$ and $b_{20}$.
However, the other six invariants are independent of
each other. Moreover,
they determine all the third-order coefficients of the
normal form \eqref{eq:west3}. Therefore, we have the following
theorem.

\begin{theorem}\label{thm:main}
Let\/ $f, g : (\R^2, \zv) \rightarrow (\R^3, \zv)$ be map-germs, and\/
$\zv$ cuspidal edges which have the same invariants\/ $\tau,
\kappa_s,\kappa_n,\kappa_c, \kappa_t $ and\/ $\kappa_i$ at\/ $\zv$
respectively, and\/ $\kappa_n (\zv) \ne0$. Then there exists a
diffeomorphism-germ\/ $\phi:(\R^2, \zv)\to(\R^2, \zv)$ and an
isometry-germ\/ $\Phi:(\R^3, \zv)\to(\R^3, \zv)$ which satisfies
$$
f(u,v)-\Phi\big(g(\phi(u,v))\big)=O(4),
$$
where\/ $O(4)=\{h(u,v): (\R^2, \zv) \rightarrow (\R^3, \zv)
\,|\,(\partial^{i+j}/\partial u^i\partial v^j)h(\zv)=\zv,\
i+j\leq 3\}$.
Using the differential of invariants,
if\/ $f$ and\/ $g$ have the same invariants\/
$\kappa_s$, $\kappa_n$, $\kappa_c$, $\kappa_t$, $\kappa_s'$
and\/ $\kappa_n'$ at\/ $\zv$,
then the same assertion holds.
\end{theorem}
\begin{proof}
By the formula for $\tau, \kappa_s,\kappa_n,\kappa_c, \kappa_t$
and\/ $\kappa_i$, if $\kappa_n\ne0$, these six values determine
all the coefficients $a_{20},a_{30},b_{20},b_{30},b_{12}$ and
$b_{03}$ in \eqref{eq:west3}. Thus we have the result.
By the same arguments, the second claim is proven by
Theorem \ref{thm:rel1}.
\end{proof}
We remark that for given real numbers $\tau,
\kappa_s,\kappa_n,\kappa_c, \kappa_t$ and $\kappa_i$, there exists a
map-germ $f$ at $\zv$ such that $\zv$ is a cuspidal edge, and its
six invariants at $\zv$ are $\tau, \kappa_s,\kappa_n,\kappa_c,
\kappa_t$ and $\kappa_i$, respectively, just by substituting these
real numbers into \eqref{eq:west3} and applying $h(u,v)=0$. For
global realization of fronts, see \cite{suykodai}. In Figure
\ref{fig:1} we have drawn surfaces that are images of map-germs
given by \eqref{eq:west3}. Invariants not specified are zero.
\vspace{-5mm}

\begin{figure}[!ht]
\centering
\begin{tabular}{ccc}
\includegraphics[width=.25\linewidth,bb=14 14 275 238]
{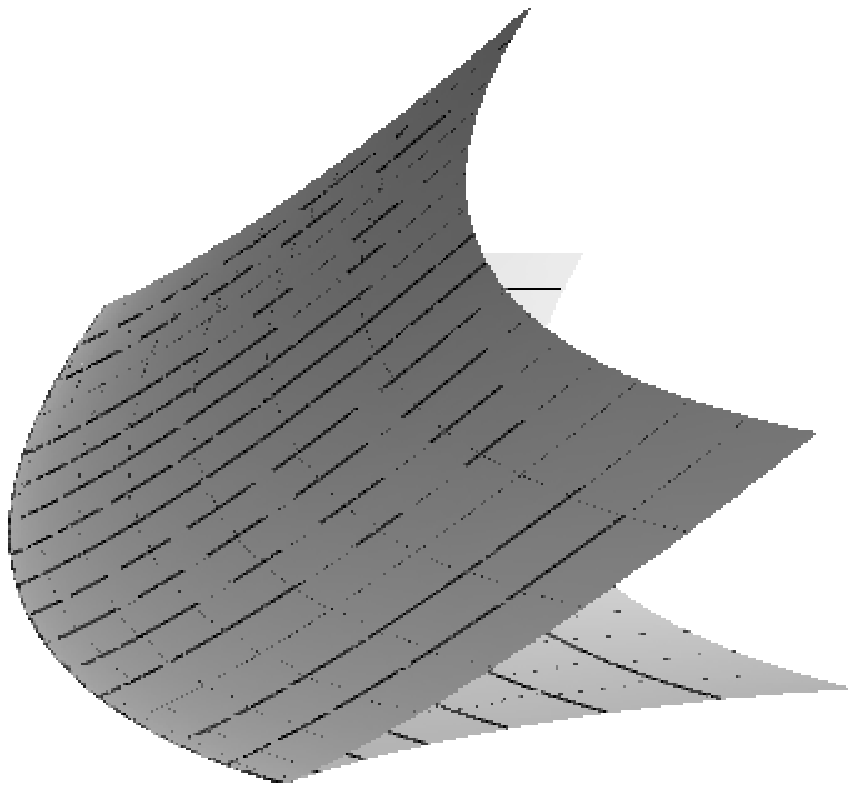} &
\includegraphics[width=.25\linewidth,bb=14 14 244 271]
{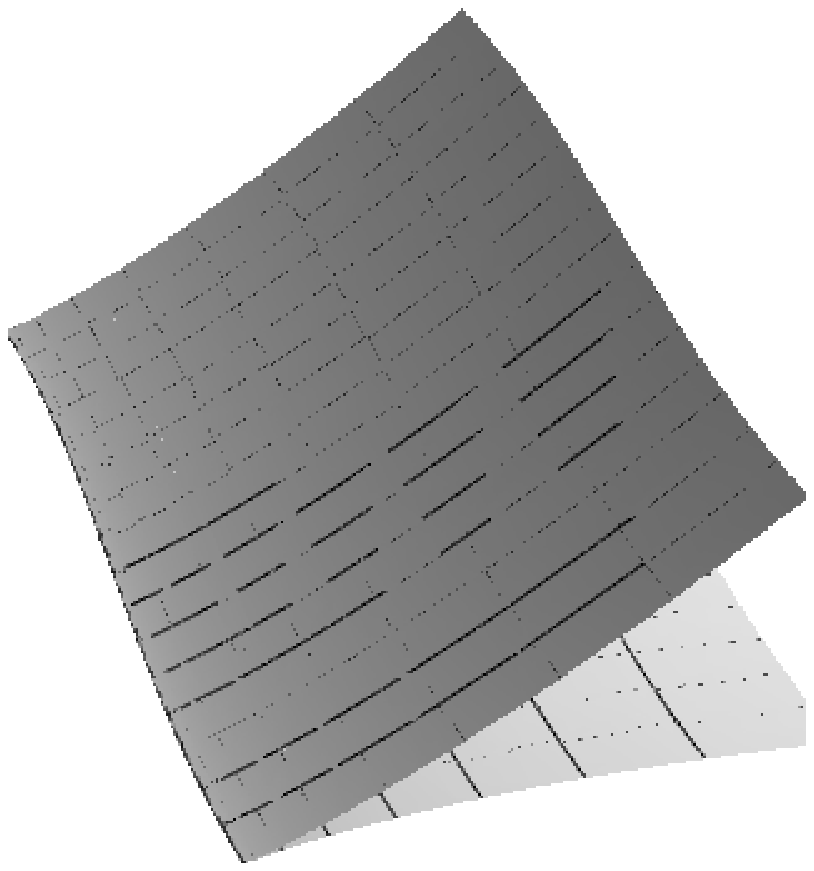} &
\includegraphics[width=.25\linewidth,bb=14 14 267 285]
{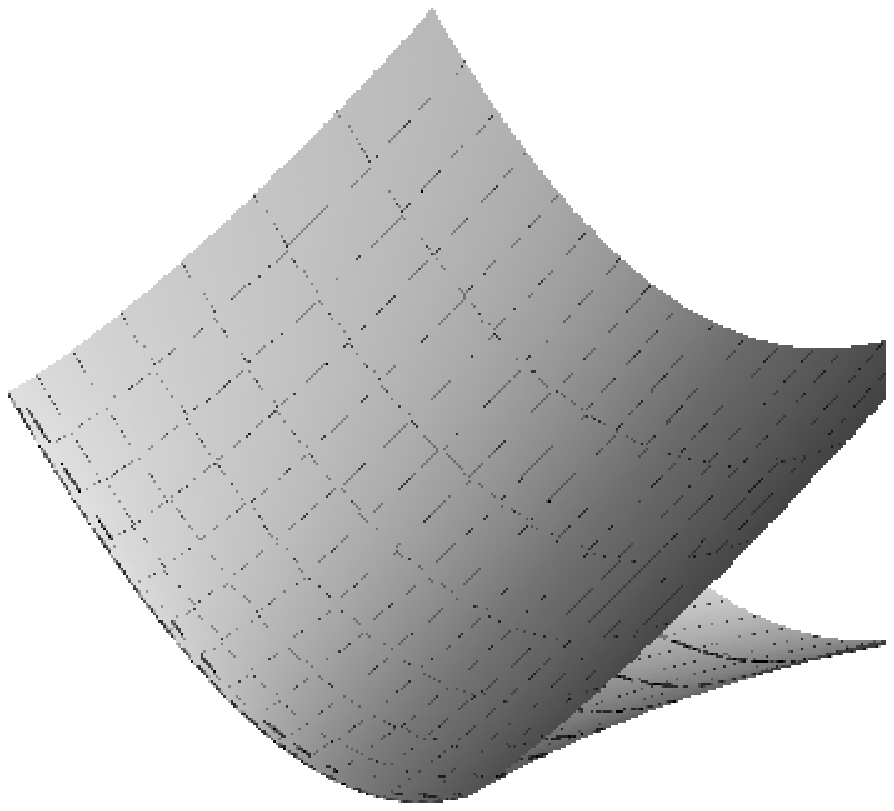}
\\
$(\kappa_s=)a_{20}=3,\ b_{03}=1$&
$a_{30}=3,\ b_{03}=1$&
$(\kappa_n=)b_{20}=3,\ b_{03}=1$\\[-5mm]
\includegraphics[width=.25\linewidth,bb=14 14 269 325]
{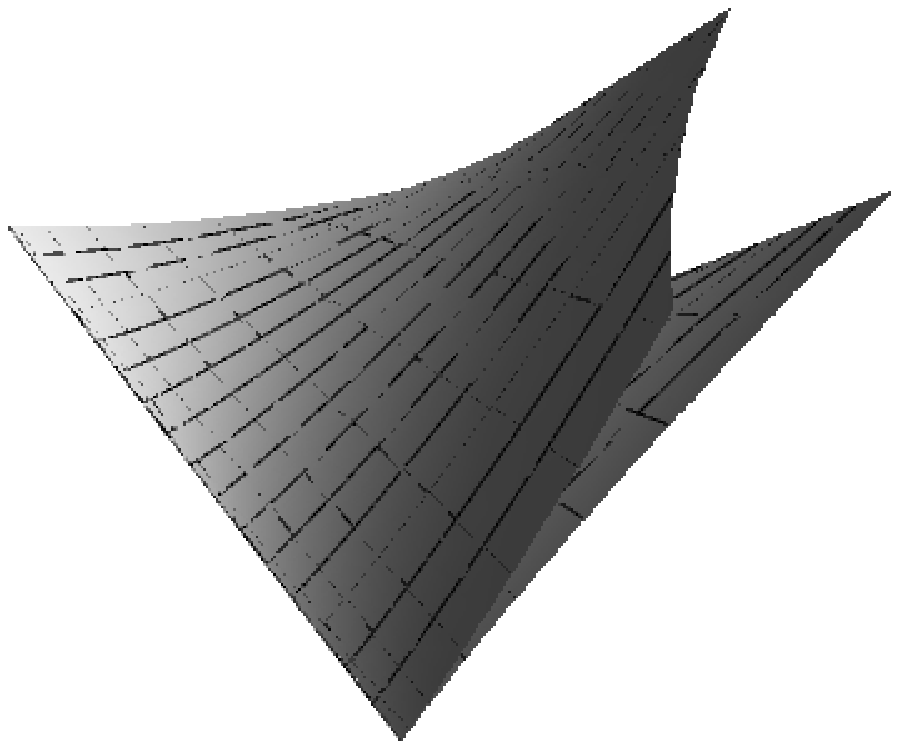} &
\includegraphics[width=.25\linewidth,bb=14 14 283 307]
{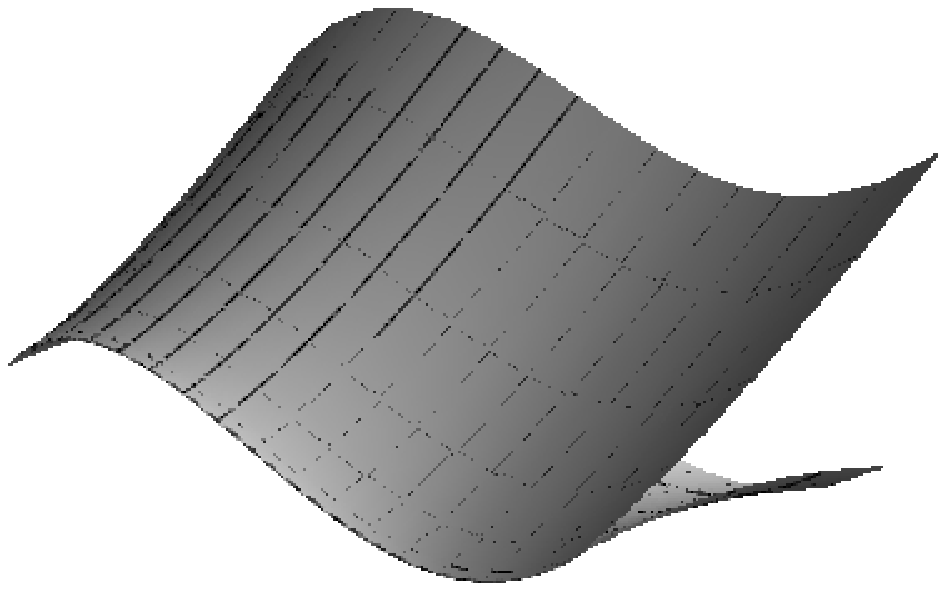} &
\includegraphics[width=.25\linewidth,bb=14 14 288 331]
{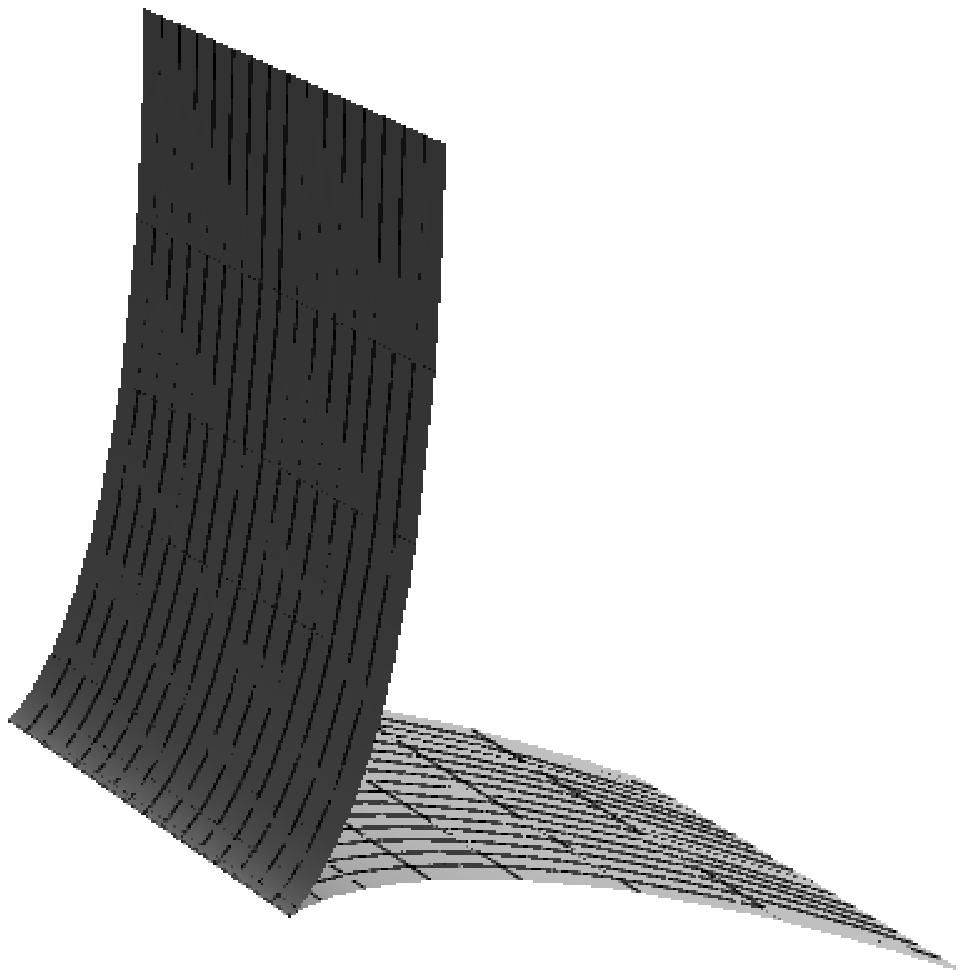}\\
$(\kappa_t=)b_{12}=3,\ b_{03}=1$&
$(\kappa_i=)b_{30}=3,\ b_{03}=1$& $(\kappa_c=)b_{03}=3$
\end{tabular}
\caption{Invariants of cuspidal edges}
\label{fig:1}
\end{figure}

\subsection{Example: Tangent developable}
Let $\hat{\gamma}: \R \rightarrow \R^3$ be a unit speed space curve
which has curvature $\kappa(u) > 0$ and torsion $\tau(u)\ne0$, for
all $u \in I$, and let $\{\vec{t}, \vec{n}, \vec{b}\}$ be the Frenet
frame. Let $f: I \times \R \rightarrow \R^3$ be given by $f(u,v) =
\hat{\gamma}(u) + v \hat{\gamma}'(u)$. Then $f$ is called a {\em
tangent developable surface} (see Figure \ref{fig:2}). Then
$S(f)=\{(u,0)\}$ and $(u,0)$ is a cuspidal edge. The unit normal
vector field is $\vec{b}(u)$, and the area density function is
proportional to $v$. Therefore, taking $\xi =
\partial_u$ and $\eta =
-\partial_u +\partial_v$,  then $(\xi, \eta)$ is an
adapted pair of vector fields.
So, by the Frenet formulas,
we have
$$
\xi f = \vec{t} + v \kappa \vec{n},\quad
\xi \xi f = \kappa \vec{n} + v(\kappa \vec{n})'\quad
\text{and}\quad
\xi \xi \xi f
=
- \kappa^2 \vec{t} + \kappa' \vec{n}
+ \kappa \tau \vec{b} + v(\kappa \vec{n})''.
$$
Thus, it holds that
$$
\xi f= \vec{t},\quad
\xi \xi f = \kappa \vec{n},\quad
\xi \xi \xi f
=
- \kappa^2 \vec{t} + \kappa' \vec{n} + \kappa \tau \vec{b}\quad
(\text{on }S(f)).
$$
Furthermore,
we have
$$\eta f = - v \kappa \vec{n},\quad
\eta \eta f = v(\kappa \vec{n})' - \kappa \vec{n},\quad
\xi \eta \eta f
=
 v(\kappa \vec{n})'' + \kappa^2 \vec{t} - \kappa' \vec{n} -
\kappa \tau \vec{b}
$$
and
$$
\eta \eta \eta f
=
-v(\kappa \vec{n})''
+ 2(- \kappa^2 \vec{t} + \kappa' \vec{n} + \kappa \tau \vec{b}).
$$
Thus, it holds that
$$
\eta \eta f = - \kappa \vec{n},\quad
\xi \eta \eta f = \kappa^2 \vec{t}
- \kappa' \vec{n} - \kappa \tau \vec{b},\quad
\eta \eta \eta f = 2(-\kappa^2 \vec{t} + \kappa' \vec{n}+ \kappa \tau
\vec{b})\quad
(\text{on }S(f)).
$$
Finally, since $\eta \lambda = v\kappa' - \kappa$,
it holds that $\sgn(\eta \lambda) =-1$.
Therefore, we obtain
$$
\kappa_s=  -\kappa(u),\
\kappa_n= 0,\
\kappa_c= - \frac{2 \tau(u)}{\sqrt{\kappa(u)}},\
\kappa_t= \tau (u),\
\kappa_i = - \kappa(u) \tau(u)\
\text{at}\ (u,0).
$$
\vspace{-20mm}

\begin{center}
\begin{figure}[ht]
\centering
\includegraphics[width=.3\linewidth,bb=14 14 298 331]
{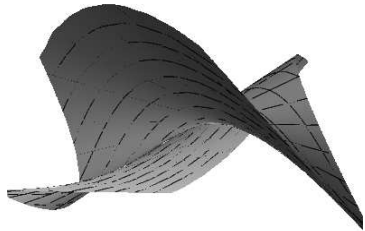}
\vspace{-10mm}

\caption{Tangent developable surface.}
\label{fig:2}
\end{figure}
\end{center}


\appendix
\section{Global property of cusp directional torsion}
In this subsection, we consider a global property of the function
$\kappa_t$.

\begin{lemma}
\label{lem:global1} Let\/ $f:(\R^2,\zv)\to(\R^3,\zv)$ be a frontal
and\/ $\zv$ a non-degenerate singularity whose singular
direction and null direction are transversal. Let $M$  be the image
of $f$. Then a slice locus\/ $M\cap N_{\zv}M$ is a curve\/
$\hat\sigma$ with $\hat\sigma'(0)=\zv$ and\/
$\hat\sigma''(0)\ne\zv$.
\end{lemma}
\begin{proof}
Let $(u,v)$ be an adapted coordinate system, since we can take it by
the assumption of $f$. Since $\inner{f(u,v)}{f_u(\zv)}_u\ne0$ at
$\zv$, there exists a function $u(v)$ ($u(0)=0$) such that
$\inner{f(u(v),v)}{f_u(\zv)}_u\equiv0$. We set $\sigma(v)=(u(v),v)$.
Then the image $\hat\sigma(v)=f\circ \sigma(v)$ coincides with the
slice locus $M\cap N_{\zv}M$. Remark that since $\partial_v$ is a
null direction, so $\inner{f(u,v)}{f_u(\zv)}_v(\zv)=0$,
and thus $u'(0)=0$ holds. By a calculation, we have
$\hat\sigma'(0)=\zv$, $\hat\sigma''(0)=f_u(\zv)u''(0)+f_{vv}(\zv)$.
Since $\zv$ is a non-degenerate singularity,
$\det(f_u,f_{vv},\nu)(\zv)\ne0$ holds. Therefore, we have
$\hat\sigma''(0)\ne\zv$.
\end{proof}
\begin{lemma}
\label{lem:global2} Let\/ $\sigma:(\R,\zv)\to(\R^2,\zv)$ be a curve
with\/ $\sigma'(0)=\zv$ and\/ $\sigma''(0)\ne\zv$. Then there exist
an orthonormal basis\/ $\{\x,\y\}$ of\/ $\R^2$ and a parameter\/ $v$
of\/ $(\R,0)$ such that\/ $j^3\sigma(v)=v^2\x/2+\alpha v^3\y/6$,
$\alpha\in\R$ holds.
\end{lemma}
\begin{proof}
Since $\sigma'(0)=\zv$ and $\sigma''(0)\ne\zv$, we can set
$j^3\sigma(v)=(a_2v^2+a_3v^3,b_2v^2+b_3v^3)$, where
$a_2,a_3,b_2,b_3\in\R$ and $(a_2,b_2)\ne0$. We can assume that
$a_2\ne0$. Set $\theta$ satisfying  $\sin\theta a_2+\cos\theta
b_2=0$. Then we have $ \tilde A_\theta \big(j^3\sigma(v)\big) =
\big( (\cos\theta a_2-\sin\theta b_2)v^2+(\cos\theta a_3-\sin\theta
b_3)v^3, (\sin\theta a_3+\cos\theta b_3)v^3 \big) $ (See
\eqref{eq:rot} for $\tilde A_\theta$.). Set $\tilde v = \sqrt{2}\,
v\big( (\cos\theta a_2-\sin\theta b_2)+v(\cos\theta a_3-\sin\theta
b_3) \big)^{1/2}$. Then $ \tilde A_\theta\big(j^3\sigma(\tilde
v)\big) = \big(\tilde v^2/2, \alpha \tilde v^3 \big)$,
$\alpha\in\R$, holds. Setting ${}^t\x=\tilde
A_\theta^{-1}\big({}^t(1,0)\big)$ and ${}^t\y=\tilde
A_\theta^{-1}\big({}^t(0,1)\big)$, we have the result.
\end{proof}
Let $\Sigma$ be a two dimensional manifold and
$f:\Sigma\to\R^3$ a frontal. Let $\gamma:S^1\to \Sigma$ be a simple
closed curve consists only of non-degenerate singularities whose
singular direction and null direction are transversal, namely
$df(\gamma')\ne\zv$.
In \cite{msuy}, this type of singularities are called
non-degenerate singular points of the {\em second kind}.
Denote $\hat{\gamma}=f\circ\gamma$.
Let $u$ be an arclength parameter of $\hat\gamma$
and
$\dd_1,\dd_2$ an orthonormal frame along $\hat\gamma$,
namely an orthonormal frame field of the normal
plane $(\hat\gamma')^\perp$ of $\hat\gamma$.
Then we have
$$
\pmt{\e'\\ \dd_1'\\ \dd_2}=
\pmt{0&c_1&c_2\\
    -c_1&0&c_3\\
    -c_2&-c_3&0}
\pmt{\e\\ \dd_1\\ \dd_2},
$$
where, $\e=\hat\gamma'$.
If the curvature of $\hat\gamma$ does not vanish
and $\e,\dd_1,\dd_2$ is the Frenet frame,
then $c_1=\kappa, c_2=0, c_3=\tau$.

For a sufficiently small $\ep$,
a map $(t_1,t_2,t_3)\mapsto\hat\gamma(t_1)+t_2\dd_1+t_3\dd_2$,
$(-\ep<t_2,t_3<\ep)$ is diffeomorphic.
Thus by Lemmas \ref{lem:global1} and \ref{lem:global2},
$f$ can be represented as
$\hat\gamma(u)+v^2\x(u)/2+\alpha v^3\y(u)/6+v^4\z(u,v)$,
where $\x(u),\y(u),\z(u,v)$ are
vector fields along $\hat\gamma(u)$
and $\{\x(u),\y(u)\}$ is an orthonormal basis
of $\hat\gamma(u)^\perp$.
Remark that $\eta\eta f$ is proportional to $\x$.
Then
there exists a function $\theta(u)$ such that
$$
\begin{array}[tb]{rcl}
    \vect{x}(u)&=&\cos\theta(u)\dd_1(u)-\sin\theta(u)\dd_2(u),\\
    \vect{y}(u)&=&\sin\theta(u)\dd_1(u)+\cos\theta(u)\dd_2(u)
\end{array}
$$
hold. By a direct calculation, $ \kappa_t(u) = c_3(u) - \theta'(u) $
holds. Hence
$\frac{1}{2\pi}\int_{\gamma}(c_3(u)-\kappa_t(u))\,du\in\Z$ holds.
This integer $n$ is called the\/ {\it intersection number of the
frame\/ $(\hat\gamma,\x)$ with respect to the frame\/
$(\hat\gamma,\dd_1)$}. Thus we have
$\int_{\gamma}\kappa_t(u)\,du=\int_{\gamma}c_3(u)\,du-2\pi n$. On
the other hand, if the curvature of $\hat\gamma$ never vanish, then
one can take $\dd_1$ as the principal normal vector. Then, $c_3$ is
the torsion of $\hat\gamma$. Hence $\int_{\gamma}\kappa_t(u)\,du$ is
equal to the difference of the total torsion of $\hat\gamma$ between
intersection number of $(\hat\gamma,\x)$  with respect to
$(\hat\gamma,\dd_1)$.
\smallskip

{\bf Acknowledgments}\hspace{3mm}
The authors thank Takashi Nishimura, Wayne Rossman,
Masaaki Umehara and Kotaro Yamada
for valuable comments.
They also thank the referee for 
careful reading and 
comments that improved the results here.
This work was partly supported by
CAPES and JSPS under Brazil-Japan research cooperative program,
Proc BEX [12998/12-5] and
Grant-in-Aid for
Scientific Research (Young Scientists (B)) [23740045], from the
Japan Society for the Promotion of Science.


\medskip
{\footnotesize
\begin{tabular}{ll}
\begin{tabular}{l}
Departamento de Matem{\'a}tica,\\
IBILCE - UNESP -
Univ Estadual Paulista \\
R. Crist{\'o}v{\~a}o Colombo, 2265, CEP 15054-000,\\
S{\~a}o Jos{\'e} do Rio Preto, SP, Brazil\\
  E-mail: {\tt lmartinsO\!\!\!aibilce.unesp.br}
\end{tabular}
&
\begin{tabular}{l}
Department of Mathematics,\\
Graduate School of Science, \\
Kobe University, \\
Rokko, Nada, Kobe 657-8501, Japan\\
  E-mail: {\tt sajiO\!\!\!amath.kobe-u.ac.jp}
\end{tabular}
\end{tabular}
}
\end{document}